\documentclass[reqno,10pt]{amsart}

\usepackage[english]{babel}
\usepackage[T1]{fontenc}

\usepackage{amsmath,amssymb,amsthm}
\usepackage{mathtools}  
\usepackage[margin=1.25in]{geometry}

\usepackage{graphics,tikz,caption,subcaption}

\usepackage{mathrsfs}
\usepackage{stmaryrd}

\usepackage[colorlinks	=	true, linkcolor	=	red, urlcolor	=	red, citecolor	=	red]{hyperref}
\usepackage{bookmark}									
\usepackage{cleveref}

\usepackage{multicol}
\usepackage{todonotes}
\usepackage{enumitem}
\usepackage{verbatim}

\theoremstyle{definition}
\newtheorem{thm}{Theorem}[section]

\newtheorem{defi}[thm]{Definition}
\newtheorem{lemm}[thm]{Lemma}

\newtheorem{prop}[thm]{Proposition}
\newtheorem{ques}[thm]{Question}

\newtheoremstyle{case}
{3pt}
{3pt}
{}
{}
{\itshape}
{:}
{.5em}
{}

\numberwithin{equation}{section}

\theoremstyle{case}

\theoremstyle{remark}

\DeclareMathOperator{\Mod}{mod}

\DeclareMathOperator{\diam}{diam}

\allowdisplaybreaks

\newcommand{\apmd}[2][]{															
	\ifthenelse{\equal{#1}{}}%
					{ \operatorname{N}_{#2}	}%
					{ \operatorname{N}_{#1}(#2) 	}}
					
\begin{document}

	\title{A doubling Loewner sphere that is not a quasisphere}
	\author{Matthew Romney}
	\address{Department of Mathematics\\ University of Hawaii at Manoa, Honolulu, HI 96822}
    \email{mromney@hawaii.edu}
	
	\date{}
    \thanks{The author is partially supported by the National Science Foundation under Grant No. DMS-2630316.}

	\maketitle
	
	\begin{abstract}
        We construct a metric $2$-sphere of finite Hausdorff $2$-measure that is doubling, linearly locally connected and Loewner but cannot be mapped to the standard sphere by a quasisymmetric mapping. Our result gives a negative answer to a question of Ntalampekos.  
	\end{abstract}

    \section{Introduction} \label{sec:introduction}

    The classical uniformization theorem of Poincar\'e and Koebe asserts that every simply connected Riemann surface is conformally equivalent to either the disk, the plane or the $2$-sphere. In recent years, a major effort has been made to extend this theorem to metric spaces satisfying minimal geometric assumptions. The \textit{uniformization problem} asks one to determine when a metric space can be mapped to a canonical model space, such as the $2$-sphere, by a geometrically well-behaved homeomorphism. At this level of generality, the class of conformal mappings is too restrictive and one instead considers more flexible classes of mappings such as \textit{quasisymmetric mappings}. A \textit{($2$-dimensional) quasisphere} is a metric space that is quasisymmetrically equivalent to the standard $2$-sphere. The prototypical result of this type is the celebrated Bonk--Kleiner theorem \cite{BK:02}.  

    \begin{thm}[Bonk--Kleiner] \label{thm:bonkkleiner}
        Let $X$ be an Ahlfors $2$-regular metric sphere. Then $X$ is a quasisphere if and only if it is linearly locally connected. 
    \end{thm}
    

    This paper concerns the following natural question.

    \begin{ques} \label{ques:generalization}
        Does there exist an intrinsic characterization of quasispheres similar to \Cref{thm:bonkkleiner} under weaker assumptions than Ahlfors $2$-regularity?
    \end{ques}

    We consider here the assumption of finite area (i.e., Hausdorff $2$-measure). This assumption is motivated by recent advances in the field showing its viability in the context of uniformization problems. Specifically, it was shown by the author and Ntalampekos in \cite{NR:21,NR:22}, and independently by Meier--Wenger \cite{MW:21} for length metrics, that any metric sphere of finite area admits a so-called \textit{weakly quasiconformal} parametrization from the standard sphere. Thus, a positive answer to \Cref{ques:generalization} could be obtained by showing that further geometric assumptions on the space allow one to ``upgrade'' a weakly quasiconformal parametrization to a quasisymmetric one. As candidate assumptions, there are three standard conditions that are necessary for a metric sphere of finite area to be a quasisphere: (1) the doubling property, (2) linear local connectedness and (3) the Loewner property. These are reviewed in \Cref{sec:background} below. 
    
    However, we show that these conditions are still insufficient, even when taken together. Thus the Bonk--Kleiner theorem does not have an analogue at this level of generality. We prove the following. 

    \begin{thm} \label{thm:main}
        A metric sphere $X$ of finite Hausdorff $2$-measure exists that is doubling, linearly locally connected and Loewner but not a quasisphere.
    \end{thm}

    This theorem provides a negative answer to a question of Ntalampekos asked in \cite[Question 1.6]{Nta:25a} and again in \cite[Question 4.7]{Nta:25b}. Note that, as observed in Thereom 1.1 of \cite{Nta:25a}, for smooth Riemannian $2$-spheres the doubling and Loewner properties are quantitatively equivalent to being a quasisphere. The same conclusion is also true for metric spheres that are \textit{reciprocal} as defined by Rajala in \cite{Raj:17}; such a space has controlled conformal modulus. Thus, any space satisfying \Cref{thm:main} must be non-reciprocal. 

    The idea of the construction is conceptually simple. There are existing examples in the literature of metric spheres that are doubling and linearly locally connected but do not satisfy the Loewner property and hence are not quasispheres. We start with such a metric sphere, denoted by $Y$; the particular example is not so essential. Here, we use a sphere constructed from dyadic slit carpets, following work of Merenkov--Wildrick \cite{MW:13} and  Hakobyan--Li \cite{HL:23}. For another source of constructions of both quasispheres and non-quasispheres, see the monograph of Bonk--Meyer \cite{BM:17}.

    Next, we modify the sphere to impose the Loewner property on it. This can be done by ``collapsing'' the metric of $Y$ on a fat Cantor set. Done naively, such a change of metric is not quasisymmetric. However, the insight of the author's results in \cite{Rom:19b} is that this collapsing can be done via a quasisymmetric deformation of the original metric. In fact, the area of the complement of the collapsed Cantor set can be made arbitrarily small yet with the change of metric having a uniform quasisymmetric control function. In this way we obtain a metric sphere that is also Loewner but not a quasisphere. The main step of our proof is a careful verification that the new metric is indeed Loewner. We note that the quasisymmetric collapsing of the metric necessarily increases the Hausdorff $2$-measure of the space, and so some care is needed in order to retain a sphere of finite area. 

    This paper is organized as follows. \Cref{sec:background} contains the necessary definitions and background. \Cref{sec:pillowcase} explains the basic pillowcase construction. In \Cref{sec:stage2}, we show how to deform this pillowcase quasisymmetrically to a Loewner space and complete the proof of \Cref{thm:main}.

    \subsection*{Acknowledgments} I thank Dimitrios Ntalampekos and Kai Rajala for their comments on a draft of this paper.

    \section{Background} \label{sec:background}

    We next review the relevant notation and definitions. In the following, $(X,d_X)$ and $(Y,d_Y)$ are metric spaces. The Euclidean metric on $\mathbb{R}^2$ is denoted by $d_{\text{Euc}}$. The open metric ball of radius $r > 0$ centered at $x \in X$ is denoted by $B_X(x,r)$, or simply $B(x,r)$. We also use $d_X$ to denote the distance between subsets: for two sets $E,F \subset X$, $d_X(E,F) = \inf\{d_X(x,y): x \in X, y \in Y\}$. The diameter of a set $E \subset X$ is $\diam_X(E) = \sup\{d_X(x,y): x,y \in X\}$.   We say that a \textit{metric sphere} is a metric space homeomorphic to the $2$-sphere.   

    \begin{defi}
        A homeomorphism $f \colon X \to Y$ is \textit{quasisymmetric} if there is a homeomorphism $\eta \colon [0,\infty) \to [0,\infty)$ such that 
    \[ \frac{d_Y(f(x),f(y))}{d_Y(f(x),f(z))} \leq \eta\left( \frac{d_X(x,y)}{d_X(x,z)} \right)\]
    for all triples of distinct points $x,y,z \in X$. The function $\eta$ is called the \textit{control function}. 
    \end{defi}

    A quasisymmetric mapping preserves relative distance between triples of points in $X$, up to the error allowed by the control function $\eta$. It is a standard exercise to show that the inverse of a quasisymmetric mapping is also quasisymmetric, as is the composition of two quasisymmetric mappings. Conformal mappings in Euclidean space satisfy the quasisymmetry property locally, and hence we may regard quasisymmetric mappings as a natural generalization to metric spaces. For the basic theory, we refer the reader to \cite[Chapters 10-11]{Hei:01}.
    
    Two metric spaces that can be mapped onto each other by a quasisymmetric homeomorphism are said to be \textit{quasisymmetrically equivalent}. A metric sphere that is quasisymmetrically equivalent to the $2$-sphere is called a \textit{quasisphere}. 

    We denote by $\mathcal{H}_X^p$, or simply $\mathcal{H}^p$, the $p$-dimensional Hausdorff measure on $X$.
    Given $p \geq 0$, the metric space $X$ is \textit{Ahlfors $p$-regular} if there exists $C \geq 1$ such that \[C^{-1} r^p \leq \mathcal{H}^p(B(x,r)) \leq Cr^p\] for all $x \in X$ and $r \in (0, \text{diam}(X))$.   


    We now define in turn the three terms appearing in \Cref{thm:main}.

    \begin{defi} 
    The metric space $X$ is \textit{doubling} if there exists $C \geq 1$ such that every open ball $B(x,r)$ in $X$ can be covered by $C$ open balls of radius $r/2$. 
    \end{defi}

    The doubling condition can be thought of as a metric notion of finite dimensionality. For example, it is easy to show that any Ahlfors $p$-regular space is doubling. 

    \begin{defi} \label{def:llc}
    The metric space $X$ is \textit{linearly locally connected} if there exists $\lambda \geq 1$ such that the following two properties hold.
    \begin{enumerate}
        \item For every ball $B(x,r) \subset X$ and points $y,z \in B(x,r)$, there exists a continuum in $B(x,\lambda r)$ connecting $y$ and $z$.
        \item For all ball $B(x,r) \subset X$ and $y,z \in X \setminus B(x,r)$, there exists a continuum in $X \setminus B(x,r/\lambda)$ connecting $y$ and $z$.
    \end{enumerate}  
    \end{defi}
    We note that if $X$ is a topological $2$-manifold, then linear local connectedness is equivalent to the property of \textit{linear local contractibility} (\cite[Lemma 2.5]{BK:02}). Linear local connectedness rules out the existence of cusps and similar features. 

    It is easy to check that both the doubling property and linear local connectedness are quasisymmetrically invariant, though the specific constants may change quantitatively. These properties are satisfied by the standard $2$-sphere and hence by any quasisphere. 

    Next, the statement of the Loewner property involves the notion of modulus of a path family, which measures the size of a path family from the point of view of quasiconformal geometry. A \textit{path} is a continuous function from an interval into $X$. We denote the image of a path $\gamma$ by $|\gamma|$ and its length by $\ell(\gamma)$, or $\ell(|\gamma|)$ if $\gamma$ is injective. Abusing notation somewhat, we often identify a path with its image in $X$. Let $\Gamma$ be a family of paths in $X$. A Borel function $\rho \colon X \to [0,\infty]$ is called \textit{admissible} for $\Gamma$ if every locally rectifiable path $\gamma \in \Gamma$ satisfies \[\int_\gamma \rho\,ds \geq 1.\]  
    The \textit{$2$-modulus}, or \textit{(conformal) modulus}, of $\Gamma$ is 
    \[\Mod \Gamma = \inf_\rho \int_X \rho^2\,d\mathcal{H}^2, \]
    where the infimum is taken over all admissible functions $\rho$. If $E,F$ are disjoint continua in $X$, we let $\Gamma(E,F)$ denote the family of all paths connecting $E$ and $F$. 


    \begin{defi}
        The metric space $X$ is \textit{Loewner} if there is a decreasing function $\varphi \colon (0,\infty) \to (0,\infty)$ such that $\Mod \Gamma(E,F) \geq \varphi(t)$ for all disjoint nondegenerate continua $E,F \subset X$ satisfying $t \leq \triangle(E,F)$.
    \end{defi}
    Here, $\triangle(E,F)$ is the relative distance
    \[\triangle(E,F) =  \frac{d_X(E,F)}{\min\{\diam(E),\diam(F)\}}.\]

    The Loewner property stipulates the existence of large families of rectifiable paths connecting disjoint continua in the space, quantitatively. It was first introduced by Heinonen--Koskela \cite{HK:98}. It is not quasisymmetrically invariant in general. For example, the Euclidean metric on $\mathbb{R}^n$ is quasisymmetrically equivalent to the \textit{snowflake metric} $d_{\text{Euc}}^\alpha$ for all $\alpha \in (0,1)$, although the latter has no nonconstant rectifiable paths and hence is not Loewner. However, it has been shown in Theorem 1.3 of \cite{Nta:25a} that any metric space of finite Hausdorff $2$-measure that is quasisymmetrically equivalent to a compact Riemannian surface must be Loewner. This fact is not obvious but relies on the uniformization results in \cite{NR:22}.


    \section{A modified dyadic slit pillowcase sphere} \label{sec:pillowcase}

    Our construction to prove \Cref{thm:main} comprises two stages. First is the construction of a metric sphere $Z$ of finite area that is doubling and linearly locally connected but not quasisymmetrically equivalent to the $2$-sphere. The construction is a \textit{pillowcase sphere} built from a \textit{dyadic slit carpet} following the main construction of Hakobyan--Li in \cite{HL:23} (using the constant sequence $\mathbf{r} = (1/2)$), though with a modification to keep the Hausdorff $2$-measure finite. See \Cref{fig:carpet} for an illustration. We refer the reader to \cite[Section 5]{HL:23} for full details of the original construction. The underlying dyadic slit carpet appeared previously in work of Merenkov--Wildrick \cite[Theorem 1.5]{MW:13} and Merenkov \cite{Mer:10}. The space $Z$ is not Loewner. In the next subsection, we will modify $Z$ quasisymmetrically to impose the Loewner condition. 

    \subsection{The underlying slit carpet}

    Given $n \in \mathbb{Z}_{\geq 0}$ and $1 \leq i,j \leq 2^n$, we define the 
    vertical slit 
    \[I_{i,j}^n = \{2^{-n}(i- 1/2)\} \times [2^{-n}(j-3/4),2^{-n}(j-1/4)].\]

    Let $W_0 = (0,1)^2 \setminus I_{1,1}^0$. Given $W_{n-1}$, we define inductively the domain
    \[W_n = W_{n-1} \setminus \bigcup_{1 \leq i,j \leq 2^n} I_{i,j}^n. \]
    We equip $W_n$ with the length metric induced by the Euclidean metric. Let $\overline{W}_n$ denote the metric completion of $W_n$. The spaces $\overline{W}_n$ converge in the Gromov--Hausdorff sense to a space $W$ called the \textit{dyadic slit carpet} generated by the sequence $\mathbf{r} = (1/2)$. By a theorem of Merenkov--Wildrick \cite{MW:13} (see also Hakobyan--Li \cite[Theorem 1.6]{HL:23}), the carpet $W$ is not quasisymmetrically equivalent to a subset of the Euclidean plane.

    Observe that each slit $I_{i,j}^n$ corresponds to a circle in the completion $\overline{W}_n$, in which each interior point of $I_{i,j}^n$ represents two points of the circle. We let $L_{i,j}^n$ denote the left arc and $R_{i,j}^n$ the right arc of this circle. 

    \begin{figure} 
    \centering
        \begin{tikzpicture}[scale=.7]
        \draw[xstep=2,ystep=2.0,red,dashed, thin] (0,0) grid (8,8);
        \draw[very thick] (0,0) to (8,0) to (8,8) to (0,8) to (0,0);
        \draw[very thick] (4,2) to (4,6);
        \draw[very thick] (2,1) to (2,3);
        \draw[very thick] (6,1) to (6,3);
        \draw[very thick] (2,5) to (2,7);
        \draw[very thick] (6,5) to (6,7);
        \draw[very thick] (1,.5) to (1,1.5);
        \draw[very thick] (3,.5) to (3,1.5);
        \draw[very thick] (5,.5) to (5,1.5);
        \draw[very thick] (7,.5) to (7,1.5);
        \draw[very thick] (1,2.5) to (1,3.5);
        \draw[very thick] (3,2.5) to (3,3.5);
        \draw[very thick] (5,2.5) to (5,3.5);
        \draw[very thick] (7,2.5) to (7,3.5);
        \draw[very thick] (1,4.5) to (1,5.5);
        \draw[very thick] (3,4.5) to (3,5.5);
        \draw[very thick] (5,4.5) to (5,5.5);
        \draw[very thick] (7,4.5) to (7,5.5);
        \draw[very thick] (1,6.5) to (1,7.5);
        \draw[very thick] (3,6.5) to (3,7.5);
        \draw[very thick] (5,6.5) to (5,7.5);
        \draw[very thick] (7,6.5) to (7,7.5);
    \end{tikzpicture}
        \caption{Three stages of the basic dyadic slit carpet construction.} \label{fig:carpet}
    \label{fig:qs}
    \end{figure}
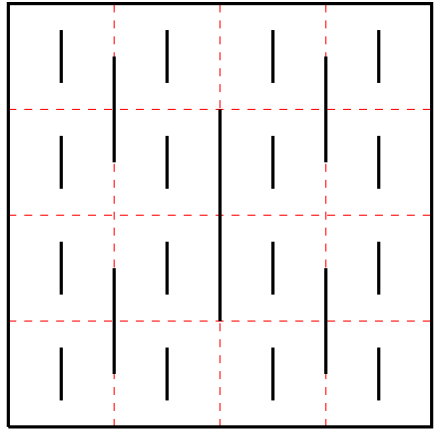

    \subsection{Adding pillowcases}

     We next add pillowcases to each space $\overline{W}_n$ to form a topological disk. Namely, to each slit $I$ we glue in a matching pillowcase according to the following procedure. Let $a>0$ denote the length of $I$. A \textit{pillowcase} is the quotient metric space formed by taking two copies $A^+, A^-$ of the square $[0,a]^2$ and gluing them together along three common sides. To be more explicit, for each choice of $\pm$, let $\varphi_\pm^1, \ldots, \varphi_\pm^4 \colon [0,a] \to A^\pm$ denote the standard arc-length parametrizations of the four sides of $A^\pm$, where the sides are parametrized cyclically (that is, $\varphi_\pm^j(a) = \varphi_\pm^{j+1}(0)$ for all $j = 1,\ldots, 4$ with the identification $\varphi_\pm^5 = \varphi_\pm^1$). Let $L,R$ denote the parametrization of the left and right arcs, respectively, of the corresponding slit $I$, from bottom to top. For each $j = 2,3,4$, we identify each pair of points $\varphi_+^j(t)$ and $\varphi_-^j(t)$. We also identify each pair of points $\varphi_-^1(t)$ and $L(t)$ and each pair of points $\varphi_+^1(t)$ and $R(t)$. Note that $A^-$ becomes the left side of the pillowcase while $A^+$ becomes the right side. 
    
    Let $Y_n$ be the length metric space formed by gluing in a pillowcase of matching side length to each slit in $\overline{W}_n$. We observe that $Y_n$ is a closed topological disk. We further note that there are $4^n$ slits in the $n$-th generation, and each pillowcase of this generation has area $2\cdot(2^{-n-1})^2 = 2^{-2n-1}$. Together, the $n$-th generation pillowcases contribute area $\frac{1}{2}$. From this we deduce that $\mathcal{H}^2(Y_n) = \mathcal{H}^2(Y_{n-1}) + \frac{1}{2} = 1 + \frac{n}{2}$. Note that these surfaces do not have uniformly bounded area, and so we require a modification to keep the area bounded. 

    \subsection{Gluing together rescaled copies of $Y_n$} \label{sec:rescaling}
    
    We choose a decreasing sequence $(\delta_n)$ that converges to $0$ sufficiently rapidly. Specifically, we choose $\delta_n \leq  K(8^{-n})^{-1}2^{-n}$, where $K = K(\beta)$ is the constant in \Cref{thm:qs} below, depending on a choice of parameter $\beta$. The value $\delta_n$ is also chosen to be a power of $1/2$ and to satisfy $\delta_n \leq 2^{-n-2}\delta_{n-1}$. 
    
    Let $Z_1 = Y_1$. Define $Z_n$ inductively by removing the square $[0,\delta_n]^2$ from $Z_{n-1}$ and replacing it with a copy of $Y_n$ rescaled by the factor $\delta_n$. Equip $Z_n$ with the induced length metric. Note that $\delta_n$ is sufficiently small so that the image of $Y_n$ in $Z_n$ in this procedure avoids any slits in $Z_{n-1}$. 
    Finally, let $Z^+$ be the Gromov--Hausdorff limit of $Z_n$ and $Z^-$ be the unit square $[0,1]^2$, which we glue together along the boundary in the standard way to form a sphere $Z$. 
    
    Let $T_n$ denote the square $[0,\delta_n]^2$ as a subset of $Z$ (including the enclosed pillowcases) and let $\psi_n$ be the similarity map from $Y_n \setminus [0,2^{-n-2}]^2$ its rescaled copy in $Z \setminus T_{n+1}$. 
    We observe that
    \[\mathcal{H}^2(T_n \setminus T_{n+1}) \leq \delta_n^2 \mathcal{H}^2(Y_n) = \delta_n^2\left(1 + \frac{n}{2} \right). \] 

    \begin{prop}
        The space $Z$ is not quasisymmetrically equivalent to the standard sphere $\mathbb{S}^2$.
    \end{prop}
    \begin{proof}
        Suppose there is a quasisymmetric mapping $f \colon Z \to \mathbb{S}^2$. Then for all $n \in \mathbb{N}$,
        we can restrict $f$ to the subset $[0,\delta_n]^2$ and rescale to obtain a quasisymmetric embedding $f_n \colon Z_n \to \mathbb{R}^2$ with uniform control function. We may further normalize so that $f_n(0,0) = (0,0)$, $f_n(1,0) = (1,0)$, $f_n(0,1) = (0,1)$, again with uniform control function. It is known (cf. \cite[Corollary 10.30]{Hei:01}) that the family $f_n$ is precompact and so has a subsequence converging to a quasisymmetric embedding from the Gromov--Hausdorff limit of $Z_n$. This restricts to a quasisymmetric embedding of the dyadic slit carpet $X$ into $\mathbb{R}^2$. As previously noted, this contradicts Theorem 1.6 of \cite{HL:23}.  
    \end{proof}

    \begin{prop}
        The space $Z$ is doubling and linearly locally connected.
    \end{prop}
    The doubling property for the underlying slit carpet $W$ is found in Proposition 2.4 of \cite{Mer:10}. The linear local connectedness of $W$ can be found as Lemma 8.4 of \cite{HL:23}. That the same properties necessarily hold for the corresponding pillowcase sphere follows from Theorem 2.6.2 in \cite{Hai:15}. Since the same properties for the space $Z$ can be established along similar lines, we omit the proof here.

    \section{Quasisymmetric deformation to a Loewner space} \label{sec:stage2}

    In the second stage, we quasisymmetrically deform the space $Z$ to be a Loewner space, denoted $X$, satisfying \Cref{thm:main}. 

    \subsection{Quasisymmetrically collapsing the metric}

    Our argument applies the following result of the author in \cite{Rom:19b}. 

    \begin{thm} \label{thm:qs}
        Let $\beta \in (0,1)$. There is a length metric $d_\beta$ on $[0,1]^2$ and a constant $K = K(\beta) > 1$, decreasing as a function of $\beta$, satisfying the following:
        \begin{itemize}
            \item[(1)] The change of metric map $\iota \colon ([0,1]^2,d_{\text{Euc}}) \to ([0,1]^2,d_\beta)$ is quasisymmetric with uniform control function $\eta$.
            \item[(2)] $\iota$ is $K$-Lipschitz.
            \item[(3)] There is a closed set $S \subset [0,1]^2$ with Lebesgue measure greater than $1-\beta^2$ such that the Hausdorff dimension of $\iota(S)$ is at most $1/2$.
            \item[(4)] $[0,1]^2 \setminus S$ is the countable union of disjoint squares (not necessarily open or closed) on which $\iota$ is locally a similarity map. 
            \item[(5)] The metric $d_\beta$ on the boundary curve $\partial [0,1]^2$ is uniformly bi-Lipschitz equivalent to the Euclidean metric.
        \end{itemize} 
    \end{thm}
    This is essentially Theorem 1.2 in \cite{Rom:19b}, with the observation that the area of $S$ can be made arbitrarily close to $1$ simply by allowing the Lipschitz constant $K$ to become large, while retaining a uniform quasisymmetric control function $\eta$. That requirement that $S$ is closed is not addressed in \cite{Rom:19b}, but it is shown in \cite{Rom:25} how to obtain a set that is closed. Properties (4) and (5) are not included in the statement of Theorem 1.2 of \cite{Rom:19b} but are part of the construction used in its proof.

    Let $\Gamma = \{\gamma_t: a \leq t \leq b\}$ be a path family indexed by the parameter $t$. Given $c \in [0,1]$, we say that \textit{$c$-proportion of paths in $\Gamma$} satisfy a given property if the set $H \subset [a,b]$ for which $\gamma_t$ satisfies the property for all $t \in H$ satisfies $\mathcal{L}^1(H)/(b-a) \geq c$. Here, $\mathcal{L}^m$ denotes the $m$-dimensional Lebesgue measure.

    \begin{lemm} \label{lemm:short_paths}
        Let $\beta \in (0,1)$. Let $\Gamma$ denote the family of horizontal paths traversing $[0,1]^2$, indexed in the usual way. Then $(1-\beta)$-proportion of the paths $\gamma_t \in \Gamma$ satisfy $\ell_{\text{Euc}}(|\gamma|\setminus S) \leq \beta$.     
    \end{lemm}
    \begin{proof}
        Let $c$ be such that $c$-propotion of $\Gamma$ does not satisfy the condition. Applying Fubini's Theorem, we have that $\mathcal{L}^2([0,1]^2 \setminus S) \geq c \beta$. Note that $\mathcal{L}^2([0,1]^2 \setminus S) \leq \beta^2$ by \Cref{thm:qs}(3), and so $c \leq \beta$. 
    \end{proof}

    
    We obtain the surface $X$ from $Z$ as follows. For each pillowcase $A \subset Z$ with faces $A^+, A^-$ and side length $a$, we replace each face $A^\pm$ with a copy of $([0,1]^2,d_\beta)$, rescaled by $a$, for $\beta$ sufficiently large as we now specify. Assume that $A$ belongs to $T_n \setminus T_{n+1}$ and $\psi_n^{-1}(A)$ has side length $2^{-k}$. Then we take $\beta = 8^{-k}$. Note that $\beta \geq 8^{-n}$. Let $X^+$ and $X^-$ be the subsets of $X$ obtained from $Z^+$ and $Z^-$, respectively. The following proposition is evident.

    \begin{prop}
        The change of metric map $\iota \colon Z \to X$ is quasisymmetric. 
    \end{prop}
    We continue to denote the image of $T_n$ in $X$ under $\iota$ by $T_n$. Next, the area of $X^+$ satisfies 
    \begin{align*}
        \mathcal{H}_X^2(X^+) & \leq \sum_{n=1}^\infty \mathcal{H}_X^2(T_n \setminus T_{n+1}) \leq \sum_{n=1}^\infty K(8^{-n})^2 \mathcal{H}_Z^2(T_n\setminus T_{n+1})  \\ & \leq \sum_{n=1}^\infty K(8^{-n})^2 \delta_n^2 \left(1+\frac{n}{2}\right) \leq \sum_{n=1}^\infty 4^{-n} \left(1+ \frac{n}{2}\right), 
    \end{align*}
    since $\delta_n \leq  K(8^{-n})^{-1}2^{-n}$. We conclude that $X$ has finite area. 

    \subsection{A modified space $\widetilde{X}$ for modulus estimates}
    
    In the remaining arguments, we will generally not work with the metric $d_X$ directly. Instead, we introduce the following modified space $\widetilde{X}$ to allow for convenient modulus estimates. This is defined similarly to $X$, but replacing the original metric on each pillowcase face $A$ by the metric $d_{A}$ defined as follows. Let $S(A)$ be the subset of $A$ arising from applying \Cref{thm:qs}. Then $d_{A}$ is the length metric given by the conformal weight $\chi_{A \setminus S(A)}$ applied to the Euclidean metric on each face. That is, we define the length of a path $\gamma$ to be $\ell(\gamma) = \int_\gamma \chi_{A \setminus S(A)}\,ds$, where $ds$ here is the Euclidean length element. See Example 2.1 in \cite{Raj:17} for a careful description of the metric arising from such a conformal weight. 

    \begin{lemm}
        The change of metric map $\widetilde{\iota} \colon X \to \widetilde{X}$ preserves conformal modulus.
    \end{lemm}
    \begin{proof}
        The map $\widetilde{\iota}$ is locally isometric except on the pillowcases. It suffices to show that $\widetilde{\iota}$ preserves modulus on each pillowcase. Let $A$ be such a pillowcase and $\Gamma$ a path family in $A$ as a subspace of $X$. Let $\widetilde{\Gamma} = \{\widetilde{\iota} \circ \gamma: \gamma \in \Gamma\}$. For a given $\gamma \in \Gamma$, observe that $|\gamma| \cap S(A)$ has zero length, and so the path integral of any Borel function $\rho$ over $\gamma$ is the same when restricted to $|\gamma| \setminus S(A)$. The same holds for $|\widetilde{\iota} \circ \gamma| \cap \widetilde{\iota}(S(A))$. Now $\widetilde{\iota}$ is locally a similarity map on the set $|\gamma| \setminus S(A)$, so we can do the usual conformal change of variables. Namely, let $\{U_i\}_{i=1}^\infty$ be an enumeration of the squares in \Cref{thm:qs}(4) on which $\widetilde{\iota}$ is a similarity map, with $\sigma_i>0$ the scaling factor corresponding to $U_i$ under $\widetilde{\iota}$. Let $\rho$ be an admissible function for $\Gamma$, and let $\rho_i = \rho\chi_{U_i}$. Let $\widetilde{\rho}_i =  \rho_i \circ \widetilde{\iota}^{-1}/\sigma_i$ and $\widetilde{\rho} = \sum_{i=1}^\infty \widetilde{\rho}_i$. Then for any path $\widetilde{\gamma} = \widetilde{\iota} \circ \gamma \in \widetilde{\Gamma}$ we have
        \[\int_{\widetilde{\gamma}} \widetilde{\rho}\,ds = \sum_{i=1}^\infty \int_{\widetilde{\gamma}} \widetilde{\rho}_i\,ds = \sum_{i=1}^\infty \int_{\gamma} (\widetilde{\rho}_i \circ\widetilde{\iota})\sigma_i\,ds = \sum_{i=1}^\infty \int_{\gamma} \rho_i\,ds = \int_{\gamma}\rho\,ds .\]
        Thus $\widetilde{\rho}$ is admissible for $\widetilde{\Gamma}$. We further compute
        \[ \int_X \rho^2\,d\mathcal{H}_X^2 = \sum_{i=1}^\infty \int_{U_i} \rho_i^2 \,d\mathcal{H}_X^2 = \sum_{i=1}^\infty \int_{U_i} (\widetilde{\rho}_i \circ \widetilde{\iota})^2 \sigma^2 d\mathcal{H}_X^2=\sum_{i=1}^\infty \int_{\widetilde{\iota}(U_i)} \widetilde{\rho}_i^2d\mathcal{H}_{\widetilde{X}}^2 = \int_{\widetilde{X}} \widetilde{\rho}^2d\mathcal{H}_{\widetilde{X}}^2.\]
        We conclude from these computations that $\Mod \widetilde{\Gamma} \leq \Mod \Gamma$. The same argument with roles reversed shows that $\Mod \Gamma \leq \Mod \widetilde{\Gamma}$. The lemma follows. 
    \end{proof}

    In addition, we will measure the relative distance $\triangle(E,F)$ of two disjoint continua using the original metric $d_Z$ rather than the quasisymmetrically equivalent metric $d_X$. Accordingly, we write $\triangle_{Z}(E,F)$ and $\triangle_{X}(E,F)$ to indicate the choice of metric. This is justified since quasisymmetric mappings distort relative distance in a controlled way: for all $t >0$, there is a corresponding $t'>0$ such that $\triangle_X(E,F) \leq t$ implies $\triangle_Z(E,F) \leq t'$. See for example Lemma 2.1 in \cite{Nta:26}. 

    \subsection{Level $n$ path families}

    We let $\pi = (\pi_1, \pi_2)$ denote the projection map from $X^+$ to $[0,1]^2$ (that is, each point in a pillowcase is mapped to the point in $[0,1]^2$ beneath it). We also let $\pi_3$ denote the vertical coordinate of a point $x$. In particular, $\pi_3(x)>0$ if and only if $x$ belongs to the interior of a pillowcase.

     For all $t \in [0,1]$, we define $\gamma_t$ to be the path whose image is the set $\pi_2^{-1}(t)$. For all $t \in [0,1]$, we define $\widetilde{\gamma}_t$ to be the path whose image is $\pi_1^{-1}(t)$. If $t$ is dyadic, there are multiple choices of paths depending on whether the left or right side of each slit is traversed, so $\widetilde{\gamma}_t$ will denote a choice of such path, with the context making clear which one. Next, let $A$ be a pillowcase, with each face identified as $[0,2^{-m}]^2$ for some $m \in \mathbb{N}$. For all $0 \leq t < 2^{-m}$, let $\theta_{A,t}$ be the simple closed path whose image in each face is $[0,2^{-m}] \times\{t\}$. Also, let $\theta_{A,2^{-m}}$ be a path that traverses the top of the pillowcase. The paths $\gamma_t$ and their subpaths are said to have \textit{type 1}, the paths $\widetilde{\gamma}_2$ and their subpaths \textit{type 2}, and the paths $\theta_{A,t}$ and their subpaths \textit{type 3}. We also say that any path of type 1,2, or 3 is \textit{in a coordinate direction}. Observe that paths in families of type 1 cross over the pillowcase corresponding to any slit it encounters, whereas paths of type 2 avoid the interior of pillowcases. 

     We pick out three base sets of path families for each $n \in \mathbb{N}$, which we collectively refer to as \textit{level $n$ path families}, as follows.
    \begin{itemize}
        \item For each $1 \leq i \leq 2^n$, the path families \[\Gamma_i^n = \{\gamma_t: (i-1)2^{-n} \leq t \leq i 2^{-n}\}.\]
        \item For each $1 \leq i \leq 2^n$, the path families \[\widetilde{\Gamma}_i^n = \{ \widetilde{\gamma}_t: (i-1)2^{-n} \leq t \leq i 2^{-n}\}.\]
        By convention, $\widetilde{\Gamma}_i^n$ contains all possible choices of $\widetilde{\gamma}_t$ in the case where $t$ is dyadic.
        \item For each pillowcase $A$ of side length $2^{-m}$, $n \geq m$, and $1 \leq i \leq 2^{n-m}$ the path families \[\Theta_{A,i}^{n} = \{\theta_{A,t} : (i-1)2^{-n} \leq t \leq i2^{-n}\}.\] 
    \end{itemize}
    We say these path families have \textit{type 1}, \textit{type 2} and \textit{type 3}, respectively. 

    Given values $0 \leq s_1 < s_2 \leq 1$, we let $\gamma_t[s_1,s_2]$ denote the subpath of $\gamma_t$ comprising those points $x \in |\gamma_t|$ with $s_1 < \pi_1(x) < s_2$ together with its limit points. In particular,  $\gamma_t[s_1,s_2]$ crosses over any pillowcases in its interior but not at its endpoints. We similarly let $\widetilde{\gamma}_t[s_1, s_2]$ denote the subpath of $\widetilde{\gamma}_t$ comprising those points $x \in |\widetilde{\gamma}|$ with $s_1 \leq \pi_2(x) \leq s_2$. Let $\theta_{A,t}[s_1^\pm, s_2^\pm]$ denote the subpath of $\theta_{A,t}$ running counterclockwise from $\pi_3^{-1}(s_1) \cap A^+$ or $\pi_3^{-1}(s_1) \cap A^{-}$ (depending on the first choice of $\pm$) to $\pi_3^{-1}(s_2) \cap A^+$ or $\pi_3^{-1}(s_2) \cap A^-$ (depending on the second choice of $\pm$). 
    Define $\Gamma_i^n[s_1, s_2]$ as the set of all paths $\gamma_t[s_1,s_2]$, where $\gamma_t \in \Gamma_i^n$, and likewise for $\widetilde{\Gamma}_i^n[s_1,s_2]$ and $\Theta_{A,i}^{n}[s_1^\pm,s_2^\pm]$. Finally, we define $\Gamma_i^n[s_1]$, $\widetilde{\Gamma}_i^n[s_1]$, and $\Theta_{A,i}^n[s_1^\pm]$ as the respective family of constant paths. Note that $\Gamma_i^n[s_1]$ is not well-defined if the path lies on a slit, since it could potentially lie on the left or right side; then $\Gamma_i^n[s_1]$ refers to a choice of one of these. If needed, we can indicate the side by writing $\Gamma_i^n[s_1^\pm]$. The notation $\Theta_{A,i}^{n}[s_1,s_2]$ and $\Theta_{A,i}^{n}[s_1]$ is interpreted similarly. 
    


    The proof of the Loewner property depends on the following estimate on the $\widetilde{X}$-length of paths of type 1. Given a type 1 path family $\Gamma_i^n[s_1,s_2]$, let $\mathcal{P}_i^n(s_1,s_2)$ denote the collection of pillowcases of side length $s_2-s_1$ or greater whose interior is intersected by some path in $\Gamma_i^n[s_1,s_2]$. Note that all pillowcases in $\mathcal{P}_i^n(s_1,s_2)$ share the same $\pi_1$-coordinate. 

    \begin{lemm} \label{lemm:type_1_estimate}
        $\frac{1}{4}$-proportion of the paths $\gamma_t \in \Gamma_i^n[s_1,s_2]$ satisfy
        \begin{equation} \label{eq:type_1}
            \ell_{\widetilde{X}}\left(|\gamma_t[s_1,s_2]| \setminus \bigcup \mathcal{P}_i^n(s_1,s_2)\right) \leq 2(n+\log_2(s_2-s_1) + 3) (s_2 - s_1). 
        \end{equation}
    \end{lemm}
    \begin{proof}
        Observe that any pillowcase crossed by some path in $\Gamma_i^n[s_1,s_2]$ that is not in $\mathcal{P}_i^n(s_1,s_2)$ must have side length less than $s_2 - s_1$. We split these remaining pillowcases into two groups. Let $\mathcal{Q}_i^n(s_1,s_2)$ denote those of side length at least $2^{-n}$, and let $\mathcal{R}_i^n(s_1,s_2)$ denote those of side length less than $2^{-n}$. Let $j = \lfloor -\log_2(s_2 - s_1) \rfloor$, noting that $j \geq 0$.

        Any individual path in $\Gamma_i^n[s_1,s_2]$ crosses at most $2^{k-j-1}$ pillowcases of side length $2^{-k}$, since these pillowcases have horizontal separation of $2^{-k+1}$. The length of its intersection with a given pillowcase is $2^{-k+1}$.
        Thus for each $\gamma_t[s_1,s_2]$, its total length contained in pillowcases from $\mathcal{Q}_i^n(s_1,s_2)$ is at most
        \[\sum_{k= j+1}^n 2^{k-j-1}\cdot 2^{-k+1} = \sum_{k= j +1}^n 2^{-j} \leq 2^{-j}(n-j) .\]
        Next, we consider those pillowcases in $\mathcal{R}_i^n(s_1,s_2)$. We observe that $\mathcal{R}_i^n(s_1,s_2)$ contains at most 
        \[2^{k-j-1}\cdot 2^{k-n} = 2^{2k-j-n-1}\]
        pillowcases of side length $2^{-k}$ for each $k>n$.

        There is at most one value $q \in \mathbb{N}$ with the property that some path in $\Gamma_i^n[s_1,s_2]$ intersects a pillowcase in the set $T_q \setminus T_{q+1}$ (defined in \Cref{sec:rescaling}), considered here as a subset of $Z$, except in the case that $i=1$. In the latter case, let $q$ be the minimal value such that some path in $\Gamma_i^n[s_1,s_2]$ intersects a slit in $T_q \setminus T_{q+1}$. By the assumptions on $\delta_n$ in \Cref{sec:rescaling}, at most $\frac{1}{4}$-proportion of the paths in $\Gamma_i^n[s_1,s_2]$ intersect pillowcases in $T_{q'}$ for $q'>q$, and thus the following analysis will consider only pillowcases in $T_q \setminus T_{q+1}$. 

        Note that \Cref{thm:qs} is applied to each pillowcase $A \in \mathcal{R}_i^n(s_1,s_2)$ in $T_q \setminus T_{q-1}$ of side length $2^{-k}$ for $\beta =\beta_{k,q}$ of the form $\beta_{k,q} = 8^{-k+r}$, where $r = -\log_2 \delta_q$. It necessarily holds that $k-r \geq k-n$, so that $\beta_{k,q} \leq 8^{-(k-n)}$. 

        Next, by \Cref{lemm:short_paths}, $(1-2\beta_{k,q})$-proportion of the paths in $\Gamma_i^n[s_1,s_2]$ intersect a given pillowcase in $\mathcal{R}_i^n(s_1,s_2)$ in a set of $\widetilde{X}$-length at most $2\beta_{k,q}2^{-k}$. Let $H_i^n(s_1,s_2)$ denote the set of values $t \in ((i-1)2^{-n},i2^{-n})$ for which the path $\gamma_t[s_1,s_2]$ intersects some pillowcase in $\mathcal{R}_i^n(s_1,s_2)$ of side length $2^{-k}$ in a set with $\widetilde{X}$-length greater than $2\beta_{k,q} 2^{-k}$. We see that $H_i^n(s_1,s_2)$ satisfies
        \begin{align*}
            \mathcal{L}^1(H_i^n(s_1,s_2)) & \leq \sum_{k = n+1}^\infty 2^{2k-j-n-1} \cdot 2\cdot  8^{-(k-n)} \cdot 2^{-k} \\
             & = 2^{-j} \sum_{k=n+1}^\infty 2^{-2k+n} = 2^{-j} \cdot \frac{4^{-(n+1)}}{1 - 4^{-(n+1)}}\cdot 2^{n} \\
             & \leq 2^{-n-1}. 
        \end{align*}
        This verifies that at most $\frac{1}{2}$-proportion of the paths in $\Gamma_i^n[s_1,s_2]$ are in $H_i^n(s_1,s_2)$. We conclude that at least $1/4$ proportion of the paths in $\Gamma_i^n[s_1,s_2]$ intersect only pillowcases in $T_q \setminus T_{g+1}$ and every pillowcase in $\mathcal{R}_i^n(s_1,s_2)$ in a set of length at most $2\beta_{k,q} 2^{-k}$. For such a path $\gamma_t$, its total $\widetilde{X}$-length of intersection with the pillowcases in $\mathcal{R}_i^n(s_1,s_2)$ is at most
        \begin{align*}
            \sum_{k=n+1}^\infty N(k) \beta_{k-j} & \leq \sum_{k=n+1}^\infty 2^{k-j-1}\cdot 2 \cdot 8^{-(k-n)}\cdot 2^{-k} \\ 
            & = \sum_{k=n+1}^\infty 2^{-3k+3n-j} \leq 2^{-j}.
        \end{align*}
        Since $2^{-j-1} < s_2-s_1 \leq 2^{-j}$, we conclude that 
        \[\ell_{\widetilde{X}}(\gamma_t[s_1,s_2]) \leq s_2 - s_1 + 2^{-j}(n-j) + 2^{-j} \leq 2(n-j+2)(s_2-s_1).\]
        Using the relation $-j < \log_2(s_2-s_1)+1$, this yields the lemma. 
    \end{proof}

        In the next lemma, we find a large level $n$ path family passing through a given continuum. For its statement, we say that $s \in \mathbb{R}$ is \textit{$j$-dyadic} if $s$ is a multiple of $2^{-j}$. 
        We also make the following definition.

        \begin{defi} \label{defi:extension}
            For a level $n$ path family $\Lambda[s_1,s_2]$ and $\varepsilon>0$, the \textit{$\varepsilon$-extension} of $\Lambda[s_1,s_2]$ is the level $n$ path family in $X^+$ of the same type formed by lengthening each path in the following way. If $\Lambda[s_1,s_2]$ has type 2 or type 3, then we extend each path at each endpoint by $\varepsilon$ if possible, or the maximal extension in that direction if not possible (because the path reaches the boundary of $X^+$ or wraps entirely around a pillowcase). 
            
            If $\Lambda[s_1,s_2]$ has type 1, there are two possible rules for how to extend on each side. Let $\gamma \in \Gamma_{i}^n[s_1,s_2]$ be the restriction of a path $\gamma' \in \Gamma_{i}^n$. If $\gamma'$ intersects a pillowcase $P$ with $s_1- \varepsilon < \pi_1(P) \leq s_1$ with side length $2\varepsilon$ or greater, then we extend $\gamma$ to the left until the base of $P$ and then by length $\varepsilon$ up the right side of $P$. If no such $P$ exists, then we extend $\gamma'$ to the left until reaching a point $w_1$ with $\pi_1(w) = s_1 - \varepsilon$ if possible, or the maximal extension otherwise. Extend on the right similarly.

            We define the $\varepsilon$-extension of the level $n$ path family $\Lambda[s_1]$ similarly.
    \end{defi}

    \begin{lemm} \label{lemm:diam}
        Let $E \subset X^+$ be a continuum satisfying $\diam_Z(E) \geq 2^{-k}$ for some $k \in \mathbb{N}$. There exist a $(k+4)$-dyadic number $s$ and a level $k+5$ path family $\Lambda[s]$ such that each path in the $2^{-k-2}$-extension of $\Lambda[s]$ intersects $E$.
    \end{lemm}
    \begin{proof}
        Let $x,y \in E$ be points satisfying $d_Z(x,y) = \diam_Z(E)$. Assume that one point, say $x$, does not belong to the interior of a pillowcase. The point $x$ belongs to the offset dyadic square \[[p2^{-k-3}+2^{-k-4},(p+1)2^{-k-3}+2^{-k-4}] \times [q2^{-k-3}+2^{-k-4},(q+1)2^{-k-3}+2^{-k-4}]\]
        for some integers $p,q$. It follows that $x$ belongs to the square 
        \[Q = [p2^{-k-3},(p+2)2^{-k-3}] \times [q2^{-k-3},(q+2)2^{-k-3}]\]
        of side length $2^{-k-2}$ with 
        \[\min\{\pi_1(x) - p2^{-k-3}, (p+2)2^{-k-3} - \pi_1(x)\} \geq 2^{-k-4}\] and \[\min\{\pi_2(x) - q2^{-k-3}, (q+2)2^{-k-3} - \pi_2(x)\} \geq 2^{-k-4}.\]
        Here, we identify $Q$ with the subset of $Z^+$ of points $w$ with $(\pi_1(w),\pi_2(w)) \in Q$ and $\pi_3(w) = 0$. 

        There is at most one pillowcase $P$ with side length $2^{-k-2}$ or greater that intersects the interior of $Q$. Such a pillowcase $P$ can potentially separate $Q \setminus P$ into two components. In this case, let $Q'$ be the closure of the component of $Q \setminus P$ containing $x$. Let $Q' = Q$ otherwise. For an arbitrary point $z \in Q'$, we can define a path $\alpha_z$ from $x$ to $z$ of length at most $3 \cdot 2^{-k-2}$ as follows. At least one of the horizontal arcs of $\partial Q$ from $\pi_1^{-1}(\pi_1(x))$ to $\pi_1^{-1}(\pi_1(z))$ avoids the pillowcase $P$ and hence does not cross over any pillowcase. Let $\alpha_2$ be such a path. Let $\alpha_1$ be the vertical path from $x$ to an endpoint of $\alpha_2$, and $\alpha_3$ the vertical path from $z$ to the other endpoint of $\alpha_2$, and let $\alpha = \alpha_1 * \alpha_2 * \alpha_3$.  

        Let $Q''$ be the union of $Q'$ with any pillowcases it encloses (which have side length length at most $2^{-k-2}$) together with the set of points $w \in P$ satisfying $(\pi_1(w),\pi_2(w)) \in Q'$ and $\pi_3(w) \leq 2^{-k-2}$ on the same side of the pillowcase as $Q'$. We see that any point $w \in Q''$ satisfies $d_Z(w,x) \leq 3 \cdot 2^{-k-2} + 2^{-k-2} = 2^{-k}$, with equality only on the boundary of $Q''$ or the top edge of $P$. We see that the point $y$ must lie outside the interior of $Q''$ or on the top edge of $P$. Thus there is a component $E'$ of $E$ in $Q''$ containing $x$ and a point $z$ on $\partial Q''$ or on the top edge of $P$. 

        One of the following options must hold: $\pi_1(z) = p2^{-k-3}$, $\pi_1(z) = (p+2)2^{-k-3}$, $\pi_2(z) = q2^{-k-3}$, $\pi_2(z) = (q+2)2^{-k-3}$, or $\pi_3(z) = 2^{-k-2}$. Assume that $\pi_3(z) = 2^{-k-2}$. Note that $z$ must belong to the pillowcase $P$. We take $\Lambda = \Theta_{A,1}^{k+5}$ and $s = (q+1)2^{-k-3}$, observing that $\pi_2(Q'' \cap P)$ necessarily contains $s$. For each $t \in [0, 2^{-k-2}]$, the set $E$ must contain a point of $p \in P$ with $\pi_3(p) = t$. It is immediate from the definition of $2^{-k-2}$-extension that each path in $\Lambda[s]$ must pass through such a point $p$. The conclusion of the lemma follows. 

        A similar argument handles each of the remaining cases from the previous paragraph.

        

        Next, assume that $x,y$ are both contained in the interior of pillowcases. Suppose that $E$ contains some other point $z$ not contained in the interior of a pillowcase. The point $p$ satisfies $d_Z(x,z) \geq 2^{-k-1}$ or $d_Z(x,y) \geq 2^{-k-1}$. Repeating the previous argument using the pair $x,z$ or $y,z$, respectively, gives the same conclusion, noticing that it is still possible to choose $\Lambda$ to be a level $k+5$ path family and $s$ to be $(k+4)$-dyadic.

        Finally, assume that $x,y$ belong to the interior of the same pillowcase $P$ with base $\{s\} \times [t_1,t_2]$, where $x,y$ belong to the same face $P'$. Identify this face with $[t_1,t_2] \times [0,t_2-t_1]$ Then the set \[ Q = [\pi_2(x) - 2^{-j-1}, \pi_2(x) + 2^{-j-1}] \times [\pi_3(x) - 2^{-j-1}, \pi_3(x) + 2^{-j-1}] \cap P'\] is all within $Z$-distance $2^{-j}$ from $x$. Hence $y$ does not belong to the interior of the set. From this we conclude that either $|\pi_2(x) - \pi_2(y)| \geq 2^{-j-1}$ or $|\pi_3(x) - \pi_3(y)| \geq 2^{-j-1}$. Follow a similar argument as the first case. 
        
        If $x,y$ belong to different sides of the same face of $P$, then the set $E$ must contain a point $z$ along the side edges of $P$. Repeating the previous case with $x,z$ or $y,z$ gives the same conclusion.
    \end{proof}


    \subsection{Grid paths} \label{sec:grid_paths}

    Next, given two points $x,y \in X^+$, we want to find a well-behaved path connecting them. The idea is that this should be a path comprising a small number of pieces, each in a coordinate direction, with controlled length. We say that a point $x$ is \textit{$j$-dyadic} if $\pi_i(x)$ is a multiple of $2^{-j}$ for each $i \in \{1,2,3\}$.


    \begin{lemm} \label{lemm:grid_path}
        Let $x,y \in X^+$ satisfy $\pi_1(x) \leq \pi_1(y)$. There is a path $\alpha_{x,y}$ from $x$ to $y$, which we call the \textit{grid path} from $x$ to $y$, that is the concatenation of at most three paths $\alpha_i$ satisfying the following:
        \begin{enumerate}
            \item Each $\alpha_i$ is in a coordinate direction. Paths of type 1 alternate with paths of type 2 or 3.
            \item If the points $x,y$ are $j$-dyadic, then the endpoints of the intermediate paths $\alpha_i$ are also $j$-dyadic.
            \item If $\alpha_i$ is type 2 or type 3, then its $Z$-length is at most $2d_Z(x,y)$. If $\alpha_i$ is type 2 and both endpoints of $\alpha_i$ lie on slits, then either both lie on the right side of their respective slit or both lie on the left.
            \item If $\alpha_i$ is type 1, then $\pi_1(|\alpha_i|) \subset [\pi_1(x), \pi_1(y)]$. Moreover, if $\alpha_i$ crosses into the interior of a pillowcase $P'$ of side length $2^{-k}$, then one of the following holds:
            \begin{enumerate}
                \item $2^{-k} \leq 2M$, where $M = \max\{ \pi_1(y) - \pi_1(x), |\pi_2(y) - \pi_2(x)| \}$.
                \item $|\alpha_i| \cap P'$ has $Z$-length at most $d_Z(x,y)$.
            \end{enumerate}
        \end{enumerate}
    \end{lemm}
    \begin{proof}

    Observe first that any type 1 path $\alpha$ with $\pi_1(|\alpha|) \subset [ \pi_1(x),\pi_1(y)]$ intersects at most one pillowcase of side length $2(\pi_1(y) - \pi_1(x))$ or greater. Moreover, the region \[[\pi_1(x), \pi_1(y)] \times [\min\{ \pi_2(x),\pi_2(y)\}, \max\{\pi_2(x), \pi_2(y)\}]\] 
    intersects at most one pillowcase of side length $2M$ or greater. 
    
    We split into two main cases. See \Cref{fig:grid} for an illustration of the various cases in this proof.
    
    \subsubsection*{Case 1.} Assume there is a large pillowcase $P$ separating the points $x$ and $y$. More precisely, we assume there is a pillowcase $P$ with base slit $\{s\} \times [t_1,t_2]$ satisfying the three conditions (i) $\pi_1(x) \leq s \leq \pi_1(y)$; (ii) $t_1 \leq \pi_2(x) \leq t_2$ and $t_1 \leq \pi_2(y) \leq t_2$; (iii) $2(\pi_1(y) - \pi_1(x)) \leq t_2 - t_1$. At most one pillowcase $P$ can satisfy these conditions.

    We further observe in this case that 
    \begin{equation} \label{eq:d_Z_P}
        \min\{\pi_2(x) + \pi_2(y) - 2t_1, 2t_2 - \pi_2(x) - \pi_2(y)\} \leq d_Z(x,y).
    \end{equation}

    \begin{itemize}[left=9pt] \setlength{\itemsep}{6pt}
        \item[] \textit{Subcase 1a. Both $x$ and $y$ are not in $P$.} Let $\alpha_1$ be the type 1 path connecting $x$ to the base of $P$. Let $\alpha_2$ be a type 3 path connecting the terminal point of $\alpha_1$ to the nearer of the two endpoints of the slit $\{s\} \times [t_1,t_2]$ along the left side of the slit, then along the right side of the slit until reaching the point $w$ along the right side of the slit with $\pi_2(w) = \pi_2(y)$. Let $\alpha_3$ be a type 1 path connecting $w$ to $y$. 

        Each of the properties (1)-(4) are straightforward to check. Properties (1) and (2) are immediate from the construction. 
        Property (3) follows from \eqref{eq:d_Z_P}. The first part of property (4) is immediate from construction. The second part follows from the fact that $\alpha$ does not cross into the interior of $P$. Any other pillowcase crossed by $\alpha$ must have side length at most $2(\pi_1(y) - \pi_1(x))$ and so (a) is satisfied.
        \item[]  \textit{Subcase 1b. $x$ is not in $P$ but $y$ is in $P$.} Let $\alpha_1$ be the type 1 path connecting $x$ to the base of $P$, then continuing up along $P$ until reaching height $\pi_3(y)$. Denote the terminal point of $\alpha_1$ by $w$. Let $\alpha_2$ be the $d_Z$-geodesic from $w$ to $y$, which is a type 3 path. Observe that $d_Z(w,y) \leq d_Z(x,y)$, and so $\alpha_2$ satisfies (3). All the other properties are immediate. 
        \item[] \textit{Subcase 1c. $x$ is in $P$ but $y$ is not in $P$.} Follow the previous procedure with roles switched.
        \item[] \textit{Subcase 1d. Both $x,y$ are in $P$ and on the same side.} Let $\alpha_1$ be the type 3 path extending from $x$ until it reaches $\pi_2(y)$; let $\alpha_2$ be the type 1 path from the terminal point of $\alpha_1$ to $y$. Properties (1)-(4) are immediate.
        \item[] \textit{Subcase 1e. Both $x,y$ are in $P$ but on different sides.} Consider the $d_Z$-geodesic $\beta$ from $x$ to $y$. Note that $\beta$ crosses over the top of $P$ or one of the sides. In the first case, let $\alpha_1$ be the type 1 path from $x$ crossing over the top of the pillowcase then continuing down until reaching $\pi_3(y)$, ending at a point $w$; let $\alpha_2$ be the type 3 path from $w$ to $y$. In the second case, let $w$ be the point on the same side as $y$ satisfying $\pi_2(w) = \pi_2(y)$ and $\pi_3(w) = \pi_3(x)$ and let $\alpha_1$ be $d_Z$-geodesic from $x$ to $w$, which is a type 3 path. In both cases, we have $\ell(\alpha_1) = d_Z(x, w) \leq d_Z(x,y)$ and $\ell(\alpha_2) = d_Z(w,y) \leq d_Z(x,y)$, which verifies (3) and (4). Properties (1) and (2) are immediate.
        \end{itemize}

            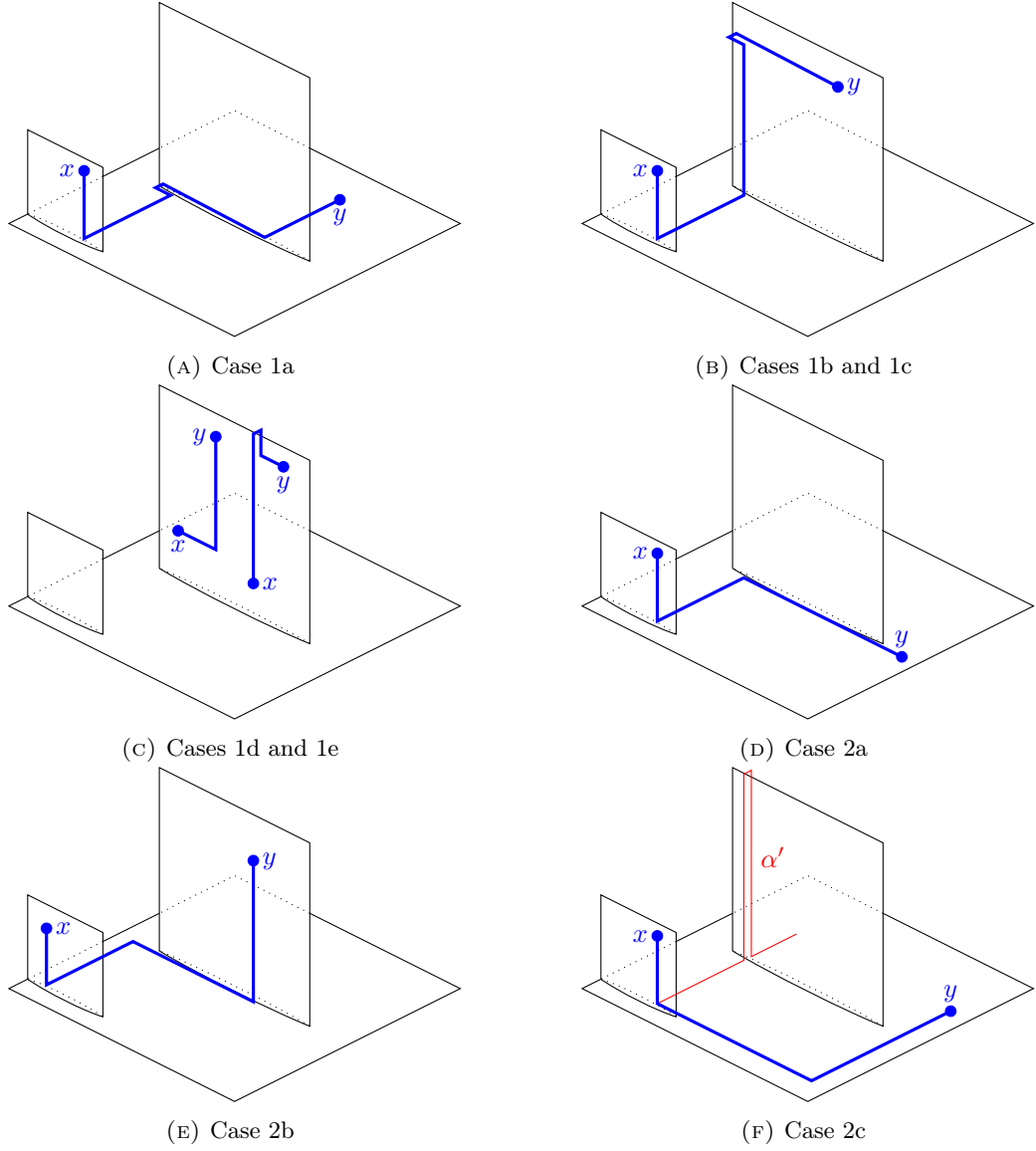
\begin{figure} 
    \centering
        \hfill
    \subfloat[Case 1a]{
        \begin{tikzpicture}[scale=1]
            \draw[] (-1,2.5) to (-1.75,2.125); 
            \draw[] (-2.75,1.625 ) to (-3,1.5) to (0,0) to (3,1.5) to (1,2.5);
            \draw[dotted] (1, 2.5) to (0,3) to (-1, 2.5);
            \draw[dotted] (-1.75,2.125) to (-2.75,1.625);
            \draw[] (-1,2) .. controls (-.5,1.7) and (.5,1.2) .. (1.,1) to (1., 3.436) to (-1,4.436) to (-1,2);
            \draw[dotted] (1.,1) .. controls (.5,1.3) and (-.5,1.8) .. (-1., 2.);
            \draw[] (-2.75,1.625) .. controls (-2.5,1.45) and (-2,1.2) .. (-1.75,1.125) to (-1.75, 2.245) to (-2.75, 2.745) to (-2.75, 1.625);
            \draw[dotted] (-2.75,1.625) .. controls (-2.5,1.55) and (-2,1.3) .. (-1.75,1.125);
            \draw[blue, very thick] (-2,2.2) to (-2,1.3) to (-.85,1.875) to (-1.05,1.975) to (-.95,2.025) to (.4,1.315) to (1.4,1.815);
            \filldraw [blue] (-2,2.2) circle (2pt) node[anchor=east]{$x$};
            \filldraw [blue] (1.4,1.815) circle (2pt) node[anchor=north]{$y$};
        \end{tikzpicture}}
    \hfill
    \subfloat[Cases 1b and 1c]{
        \begin{tikzpicture}[scale=1]
            \draw[] (-1,2.5) to (-1.75,2.125); 
            \draw[] (-2.75,1.625 ) to (-3,1.5) to (0,0) to (3,1.5) to (1,2.5);
            \draw[dotted] (1, 2.5) to (0,3) to (-1, 2.5);
            \draw[dotted] (-1.75,2.125) to (-2.75,1.625);
            \draw[] (-1,2) .. controls (-.5,1.7) and (.5,1.2) .. (1.,1) to (1., 3.436) to (-1,4.436) to (-1,2);
            \draw[dotted] (1.,1) .. controls (.5,1.3) and (-.5,1.8) .. (-1., 2.);
            \draw[] (-2.75,1.625) .. controls (-2.5,1.45) and (-2,1.2) .. (-1.75,1.125) to (-1.75, 2.245) to (-2.75, 2.745) to (-2.75, 1.625);
            \draw[dotted] (-2.75,1.625) .. controls (-2.5,1.55) and (-2,1.3) .. (-1.75,1.125);
            \draw[blue, very thick] (-2,2.2) to (-2,1.3) to (-.85,1.875) to (-.85,3.875) to (-1.05,3.975) to (-.95,4.025) to (.4,3.315);
            \filldraw [blue] (-2,2.2) circle (2pt) node[anchor=east]{$x$};
            \filldraw [blue] (.4,3.315) circle (2pt) node[anchor=west]{$y$};
        \end{tikzpicture}}
    \hfill 

    \hfill 
    \subfloat[Cases 1d and 1e]{
        \begin{tikzpicture}[scale=1]
            \draw[] (-1,2.5) to (-1.75,2.125); 
            \draw[] (-2.75,1.625 ) to (-3,1.5) to (0,0) to (3,1.5) to (1,2.5);
            \draw[dotted] (1, 2.5) to (0,3) to (-1, 2.5);
            \draw[dotted] (-1.75,2.125) to (-2.75,1.625);
            \draw[] (-1,2) .. controls (-.5,1.7) and (.5,1.2) .. (1.,1) to (1., 3.436) to (-1,4.436) to (-1,2);
            \draw[dotted] (1.,1) .. controls (.5,1.3) and (-.5,1.8) .. (-1., 2.);
            \draw[] (-2.75,1.625) .. controls (-2.5,1.45) and (-2,1.2) .. (-1.75,1.125) to (-1.75, 2.245) to (-2.75, 2.745) to (-2.75, 1.625);
            \draw[dotted] (-2.75,1.625) .. controls (-2.5,1.55) and (-2,1.3) .. (-1.75,1.125);
            \draw[blue, very thick] (-.75,2.5) to (-.25,2.25) to (-.25, 3.75);
            \filldraw [blue] (-.75,2.5) circle (2pt) node[anchor=north]{$x$};
            \filldraw [blue] (-.25, 3.75) circle (2pt) node[anchor=east]{$y$};
            \draw[blue, very thick] (.25,1.8) to (.25,3.786) to (.35,3.836) to (.35,3.5) to (.65, 3.35);
            \filldraw [blue] (.25,1.8) circle (2pt) node[anchor=west]{$x$};
            \filldraw [blue] (.65, 3.35) circle (2pt) node[anchor=north]{$y$};
        \end{tikzpicture}}
    \hfill
    \subfloat[Case 2a]{
        \begin{tikzpicture}[scale=1]
            \draw[] (-1,2.5) to (-1.75,2.125); 
            \draw[] (-2.75,1.625 ) to (-3,1.5) to (0,0) to (3,1.5) to (1,2.5);
            \draw[dotted] (1, 2.5) to (0,3) to (-1, 2.5);
            \draw[dotted] (-1.75,2.125) to (-2.75,1.625);
            \draw[] (-1,2) .. controls (-.5,1.7) and (.5,1.2) .. (1.,1) to (1., 3.436) to (-1,4.436) to (-1,2);
            \draw[dotted] (1.,1) .. controls (.5,1.3) and (-.5,1.8) .. (-1., 2.);
            \draw[] (-2.75,1.625) .. controls (-2.5,1.45) and (-2,1.2) .. (-1.75,1.125) to (-1.75, 2.245) to (-2.75, 2.745) to (-2.75, 1.625);
            \draw[dotted] (-2.75,1.625) .. controls (-2.5,1.55) and (-2,1.3) .. (-1.75,1.125);
            \draw[blue, very thick] (-2,2.2) to (-2,1.3) to (-.85,1.875) to (1.25,.825);
            \filldraw [blue] (-2,2.2) circle (2pt) node[anchor=east]{$x$};
            \filldraw [blue] (1.25,.825) circle (2pt) node[anchor=south]{$y$};
        \end{tikzpicture}}
        \hfill

            \hfill 
    \subfloat[Case 2b]{
        \begin{tikzpicture}[scale=1]
            \draw[] (-1,2.5) to (-1.75,2.125); 
            \draw[] (-2.75,1.625 ) to (-3,1.5) to (0,0) to (3,1.5) to (1,2.5);
            \draw[dotted] (1, 2.5) to (0,3) to (-1, 2.5);
            \draw[dotted] (-1.75,2.125) to (-2.75,1.625);
            \draw[] (-1,2) .. controls (-.5,1.7) and (.5,1.2) .. (1.,1) to (1., 3.436) to (-1,4.436) to (-1,2);
            \draw[dotted] (1.,1) .. controls (.5,1.3) and (-.5,1.8) .. (-1., 2.);
            \draw[] (-2.75,1.625) .. controls (-2.5,1.45) and (-2,1.2) .. (-1.75,1.125) to (-1.75, 2.245) to (-2.75, 2.745) to (-2.75, 1.625);
            \draw[dotted] (-2.75,1.625) .. controls (-2.5,1.55) and (-2,1.3) .. (-1.75,1.125);
            \draw[blue, very thick] (-2.5,2.3) to (-2.5,1.55) to (-1.35,2.125) to (.25,1.325) to (.25,3.2);
            \filldraw [blue] (-2.5,2.3) circle (2pt) node[anchor=west]{$x$};
            \filldraw [blue] (.25,3.2) circle (2pt) node[anchor=west]{$y$};
        \end{tikzpicture}}
    \hfill
    \subfloat[Case 2c]{
        \begin{tikzpicture}[scale=1]
            \draw[] (-1,2.5) to (-1.75,2.125); 
            \draw[] (-2.75,1.625 ) to (-3,1.5) to (0,0) to (3,1.5) to (1,2.5);
            \draw[dotted] (1, 2.5) to (0,3) to (-1, 2.5);
            \draw[dotted] (-1.75,2.125) to (-2.75,1.625);
            \draw[] (-1,2) .. controls (-.5,1.7) and (.5,1.2) .. (1.,1) to (1., 3.436) to (-1,4.436) to (-1,2);
            \draw[dotted] (1.,1) .. controls (.5,1.3) and (-.5,1.8) .. (-1., 2.);
            \draw[] (-2.75,1.625) .. controls (-2.5,1.45) and (-2,1.2) .. (-1.75,1.125) to (-1.75, 2.245) to (-2.75, 2.745) to (-2.75, 1.625);
            \draw[dotted] (-2.75,1.625) .. controls (-2.5,1.55) and (-2,1.3) .. (-1.75,1.125);
            \draw[red] (-2,1.3) to (-.85,1.875) to (-.85,4.35) to (-.75, 4.4) to (-.75,1.925) to (-.15,2.225);
            \filldraw[red] (-.75,3.25) circle (0pt) node[anchor=west]{$\alpha'$};
            \draw[blue, very thick] (-2,2.2) to (-2,1.3) to (.05,.275) to (1.9,1.2);
            \filldraw [blue] (-2,2.2) circle (2pt) node[anchor=east]{$x$};
            \filldraw [blue] (1.9,1.2) circle (2pt) node[anchor=south]{$y$};
        \end{tikzpicture}}
        \hfill
    \caption{Schematic diagram of the various cases of grid paths in \Cref{lemm:grid_path}} \label{fig:grid}
    \end{figure}

    \subsubsection*{Case 2.} Assume there is no such pillowcase as in the first case. We let $\alpha'$ be the type 1 path extending rightward from $x$ until reaching a point $w_1$ with $\pi_1(w_1) = \pi_1(y)$ and $\pi_3(w_1) = 0$.
    \begin{itemize}[left=9pt] \setlength{\itemsep}{6pt}
        \item[] \textit{Subcase 2a.} Assume that the terminal point of $\alpha'$ lies on the base of a pillowcase $P$ of side length greater than $2M$. Let $\alpha_1 = \alpha'$; observe that $\alpha_1$ ends at the base of $P$ and does not intersect any other pillowcase of side length greater than $2M$. The assumptions imply that $y$ does not belong to any pillowcase. Let $\alpha_2$ be a path of type 2 that extends from $w_1$ until reaching the point $y$. 
        \item[] \textit{Subcase 2b.} Assume that $\alpha'$ does not cross over the top of, or end at, a pillowcase with side length greater than $2M$. If $y$ belongs to the interior of a pillowcase $P'$, let $\alpha_3$ be the type 1 path connecting $y$ to the nearer of the two base sides of $P'$. If $\alpha_3$ ends on the left side of $P'$, then let $\alpha_1 = \alpha'$ and let $\alpha_2$ be the type 2 path connecting the endpoints of $\alpha_1$ and $\alpha_3$. On the other hand, if $\alpha_3$ ends on the right side of $P'$ and $\alpha'$ ends at the base of a pillowcase $P$, then let $\alpha_1$ be the extension of $\alpha'$ crossing over the pillowcase $P$ until reaching the base on the right side. Note that, by assumption, $P$ has side length at most $2M$. If $\alpha'$ does not end on a pillowcase, then take $\alpha_1 = \alpha'$. Again let $\alpha_2$ be the type 2 path connecting the endpoints of $\alpha_1$ and $\alpha_3$. Properties (1)-(3) are immediate. Regarding (4), it is possible that $x$ (resp. $y$) belongs to a pillowcase $P''$ of side length greater than $2M$, in which case $\alpha_1$ (resp. $\alpha_3$) crosses through $P''$ and the bound in (b) applies. Otherwise, (a) is valid by assumption. 

        If $y$ does not belong to the interior of a pillowcase, follow the same procedure without $\alpha_3$.
        \item[] \textit{Subcase 2c.} If the condition of the previous subcase is not satisfied, then let $\alpha_1$ be the type 1 path extending leftward from $y$ until reaching a point $w_1$ with $\pi_1(w_1) = \pi_1(x)$. By assumption, $\alpha'$ passes through a pillowcase $P$ with side length greater than $2M$. The existence of $P$ precludes the existence of a pillowcase $P'$ with side length greater than $2M$ through which $\alpha_1$ passes. Define $\alpha_2, \alpha_3$ as in the previous case with roles reversed. \qed
    \end{itemize}
    \renewcommand{\qedsymbol}{}
    \end{proof}

    

\subsection{Proof of the Loewner property}
    
    The proof of \Cref{thm:main} will be complete once we establish the following. 

    \begin{prop} \label{prop:loewner}
        The space $X$ is Loewner. 
    \end{prop}

    We now fix a value $t>0$. Consider two disjoint nondegenerate continua $E,F \subset X$ with $\triangle_Z(E,F) \leq t$. Assume for convenience that $E,F$ both belong to $X^+$. 
    
    We set parameters $j,k \in \mathbb{Z}$ to satisfy
    \[2^{-k-1} < d_Z(E,F) \leq 2^{-k}. \]
    and 
    \begin{equation*} \label{eq:diam}
       2^{-j-1} < \min\{\diam_Z(E), \diam_Z(F)\} \leq 2^{-j}. 
    \end{equation*}

    Observe that $j$ and $k$ are related by $2^{-k} \leq t 2^{-j+1}$. For convenience, we restrict the size of $E$ and $F$ as follows. Let $x' \in E$ and $y' \in F$ be points satisfying $d(x',y') = d(E,F)$. The set $\overline{B_Z(x',2^{-j-1})} \cap E$ has a component of diameter at least $2^{-j-1}$ containing $x'$, and similarly for $y'$. Replace $E$ and $F$ with these respective components, which we still denote by $E$ and $F$, observing that $d_Z(E,F)$ is unchanged. We now have that $2^{-j-1} \leq \diam_Z(E) \leq 2^{-j}$ and $2^{-j-1} \leq \diam_Z(F) \leq 2^{-j}$.
    



    \begin{lemm} \label{lemm:chain}
        There exist a chain of path families $\Lambda_0, \Lambda_1, \ldots, \Lambda_{I+1}$ of level $j+7$, $I \leq 3$, where each path in $\Lambda_1$ intersects $E$, each path in $\Lambda_{I+1}$ intersects $F$, and for all $1 \leq i \leq I+1$ each path from $\Lambda_{i-1}$ intersects each path from $\Lambda_{i}$. Moreover, there is a function $\nu\colon (0,\infty) \to (0, \infty)$, independent of $E$ and $F$,  such that $\frac{1}{4}$-proportion of each path family $\Lambda_i$ has $\widetilde{X}$-length at most $\nu(t)2^{-j}$.
    \end{lemm}

    
    \begin{proof}
    By \Cref{lemm:diam}, there is a $(j+5)$-dyadic number $s_0$ and level $j+6$ path family $\Lambda_0'[s_0]$ such that each path in the $2^{-k-3}$-extension of $\Lambda_0'[s_0]$ intersects $E$. Likewise, there is a $(j+5)$-dyadic number $s_{I+1}$ and level $j+6$ path family $\Lambda_{I+1}'[s_{I+1}]$ such that the same holds for the continuum $F$. Let $x$ be the $(j+7)$-dyadic point in the center of $\Lambda_0'[s_0]$ and let $y$ be the $(j+7)$ dyadic point in the center of $\Lambda_{I+1}'[s_{I+1}]$. For $i =0,I+1$, let $\Lambda_i$ be the $2^{-k-3}$-extension of either half of $\Lambda_i'[s_i]$ as separated by $x$ or $y$. Note that $\Lambda_0$ and $\Lambda_{I+1}$ are level $j+7$ path families. Let $\alpha = \alpha_1 * \cdots * \alpha_I$ ($I \leq 3$) be a grid path connecting $x$ and $y$ as given by \Cref{lemm:grid_path}. 
    Observe that 
    \[d_Z(x,y) \leq d_Z(E,F) + \diam_Z(E) + \diam_Z(F) \leq d_Z(E,F) + 2^{-j+1} \leq (2t+2)2^{-j}.\]


    Suppose that $\Lambda_0$ is of type $1$. If the left-side extension encounters a large pillowcase $P$ as in \Cref{defi:extension}, then the right-side extension does not encounter any such pillowcase. Then any path $\gamma \in \Lambda_0$ satisfies $[\pi_1(x), \pi_1(x) + 2^{-j-3}] \subset \pi_1(\gamma)$. The same applies in reverse if the right-side extension encounters a large pillowcase $P$ as in \Cref{defi:extension}.


    Next, for each path $\alpha_i$ comprising $\alpha$, $1 \leq i \leq I$, we define a corresponding level $j+7$ path family $\Lambda_{i}$ running in the same direction as $\alpha_i$ as follows. Recall that the endpoints of $\alpha_i$ are $(j+7)$-dyadic by \Cref{lemm:grid_path}. In the first case, suppose that $\alpha_i$ has type $1$. Then $\alpha_i = \beta_1 * \gamma_t[s_1,s_2] * \beta_2$ for some $t = m_i2^{-j-7}$, $m_i \in \mathbb{N}$, and $s_1 < s_2$, where $\beta_1$ and $\beta_2$ are either contained in a pillowcase or constant. In particular, $\alpha_i$ belongs to both the path families $\Gamma_{m_i-1}^{j+7}[s_1,s_2]$ and $\Gamma_{m_i}^{j+7}[s_1,s_2]$. We take $\Lambda_i$ to be the $2^{-j}$-extension of one of these families, chosen according to the following rule. Note that $\alpha_i$ intersects at most one pillowcase $P$ of side length $2|s_2-s_1|$ or greater. By construction, $\alpha_i$ does not cross over any such pillowcase, so $\alpha_i$ can either intersect an endpoint of the base slit of $P$ or end in the interior of $P$. If $\alpha_i$ intersects an endpoint of the base of $P$ in its interior, choose $\Lambda_i$ to avoid crossing over $P$. If $\alpha_i$ ends (resp. begins) at an endpoint of the base of $P$ and is followed (resp. preceded) by a type $2$ path, then choose $\Lambda_i$ so its paths do not extend into $P$. If $\alpha_i$ ends in $P$ and is followed by a type $3$ path, then choose $\Lambda_i$ so its paths do extend into $P$. In all other situations $\Lambda_i$ can be chosen arbitrarily. Our choice of $\Lambda_i$ guarantees that its paths extend to the right past $s_2 + 2^{-j}$ unless $\alpha_i$ is followed by a type 3 path in $P$, and extend to the left past $s_2 - 2^{-j}$ unless $\alpha_i$ is preceded by a type 3 path.

    In the second case, suppose that $\alpha_i$ has type $2$. Then $\alpha_i = \widetilde{\gamma}_t[s_1,s_2]$ for some $t = m_i2^{-j-7}$ and $s_1 < s_2$. In particular, $\alpha_i$ belongs to both the path families $\widetilde{\Gamma}_{m_i-1}^{j+7}[t_1,t_2]$ and $\widetilde{\Gamma}_{m_i}^{j+7}[t_1,t_2]$. We define $\Lambda_i$ to be the $2^{-j}$-extension of one of these families, chosen according to the following rule. If an endpoint of $\alpha_i$ lies on the left side of a slit, then take $\widetilde{\Gamma}_{m_i-1}^{j+7}[t_1,t_2]$. If an endpoint of $\alpha_i$ lies on the right side of a slit, then take $\widetilde{\Gamma}_{m_i}^{j+7}[t_1,t_2]$. Note from \Cref{lemm:grid_path} that if both endpoints lie on a slit, they lie on the same side, so this choice is well-defined. 

    In the third case, suppose that $\alpha_i$ has type $3$. Then $\alpha_i = \theta_{A,t}[s_1^\pm,s_2^\pm]$ for some $t = m_i2^{-j-7}$ and $s_1,s_2$. Here, we simply take $\Lambda_i$ to be the $2^{-j}$-extension of $\Theta_{A,m_i}^{j+7}[s_1^\pm, s_2^\pm]$. 
    
    In all cases, each path in $\Lambda_{i-1}$ intersects all the paths in $\Lambda_{i}$.
    



    Next, we consider the possibility that $\Lambda_0$ and $\alpha_1$ have the same type. In this case, each path in $\Lambda_1$ intersects the set $E$, and so we omit the path family $\Lambda_0$ from the construction and adjust the indices accordingly. We do the same if $\Lambda_{I+1}$ and $\alpha_I$ have the same type. 

    To complete the proof we show that $\frac{1}{4}$-proportion of the paths in $\Lambda_i$ satisfy the required bound on $\widetilde{X}$-length. 
    If $\Lambda_i$ is a path family of type 2 or type 3, then by \Cref{lemm:grid_path}(3) each path $\gamma \in \Lambda_i$ satisfies
    \[\ell_Z(\gamma) \leq 2d_Z(x,y) +2\cdot 2^{-j} \leq 2d_Z(E,F) + 4\cdot 2^{-j} \leq (4t + 4)2^{-j}.\]

    If $\Lambda_i$ has type 1, we apply \Cref{lemm:type_1_estimate} as follows. Recall that there is at most one pillowcase $P'$ of side length $s_2-s_1 + 2^{-j+1}$ or greater that is crossed into by paths in $\Lambda_i$, which is removed from the inequality in \Cref{lemm:type_1_estimate}. By construction, if such a pillowcase $P'$ occurs in the $2^{-j}$-extension portion of $\Lambda_i$ then it must satisfy $\ell_{\widetilde{X}}(|\gamma| \cap P') \leq 2^{-j+1}$. Otherwise, $P'$ must occur in the region $\pi_1^{-1}[s_1,s_2]$, and according to (4) of \Cref{lemm:grid_path} must have side length at most 
    \[2\max\{s_2-s_1, |\pi_2(y) - \pi_2(x)|\} \leq 2d_Z(x,y) \leq (4t+4)2^{-j}.\] 
    Thus for any path $\gamma \in \Lambda_i$, $|\gamma| \cap P'$ has length at most $(4t+4)2^{-j}$. Combining this bound with \Cref{lemm:type_1_estimate}, we have for $\frac{1}{4}$-proportion of $\gamma \in \Lambda_i$
    \[\ell_{\widetilde{X}}(\gamma) \leq 2(j+7 + \log_2(s_2-s_1+2^{-j+1})+3)(s_2-s_1) + (4t+4)2^{-j}.\]
    Recall that $s_2-s_1 \leq d_Z(x,y) \leq (2t+2)2^{-j}$. Thus \[\log_2(s_2-s_1+2^{-j+1}) \leq \log_2((2t+4)2^{-j})\leq \log_2 (2t+4) -j.\] It follows that for $\frac{1}{4}$-proportion of $\gamma \in \Lambda_i$
    \[\ell_{\widetilde{X}}(\gamma) \leq 2(10 + \log_2(2t+4)(2t+2)2^{-j} + (4t+4) 2^{-j}.\]
    Letting $\nu(t) = 2(10 + \log_2(2 t+4)(2t+2) + 4t+4$ gives the lemma. 
    \end{proof}

        We say that a family of paths $\Gamma = \{\gamma_t: a \leq t \leq b\}$ is \textit{rectangular} if every Borel subset $B \subset [a,b]$ satisfies
    \[ \int \rho \chi_{\Gamma(B)}\,d\mathcal{H}^2 = \int_B \int_{\gamma_t} \rho\,ds\,d\mathcal{L}^1\]
    for all Borel functions $\rho$, where $\Gamma(B)$ is the trace of $\Gamma_B = \{\gamma_t: t \in B\}$. It is easy to see that all level $n$ path families are rectangular in the metric $d_{\widetilde{X}}$. This follows from the fact that these path families foliate their trace in a given pillowcase $A$ in the standard way and the metric $d_{\widetilde{X}}$ is locally Euclidean outside $S(A)$. The following lemma will complete the proof of \Cref{prop:loewner}.

    \begin{lemm}
        For the continua $E,F$ as defined, $\Mod \Gamma(E,F) \geq \varphi(t)$.
    \end{lemm}
    \begin{proof}
        Let $\rho$ be an admissible function for the path family $\Gamma(E,F)$. Let $(\Lambda_i)$ be the sequence of level $j+7$ path families from \Cref{lemm:chain}. Observe that $\int_{\gamma} \rho\,ds \geq 1/5$ for every $\gamma \in \Lambda_i$ for some $i \in \{1, \ldots, 5\}$. Otherwise, each $\Lambda_i$ would contain a path $\gamma_i$ with $\int_{\gamma_i} \rho\,ds < 1/5$. We can concatenate subpaths of the $\gamma_i$ into a single path $\gamma \in \Gamma(E,F)$ satisfying $\int_{\gamma} \rho\,ds <1$, contradicting the admissibility of $\rho$. 
        
        Recall that $\Lambda_i$ is a rectangular path family, which we index so that $0 \leq t \leq b$ for $b = 2^{-j-7}$. Let $B$ be the subset of $[0,b]$ for which the paths $\gamma_t$ have $\widetilde{X}$-length at most $\nu(t) 2^{-j}$; then $\mathcal{L}^1(B) \geq b/4$. Thus we have
        \begin{align*}
            \frac{1}{5} \cdot \frac{b}{4} & \leq \int_0^{b} \int_{\gamma_t} \rho\,ds\,dt  = \int \rho\,d\mathcal{H}^2 
             \leq \left(\int \rho^2 d\mathcal{H}^2 \right)^{1/2} \left( \int \chi_B d\mathcal{H}^2 \right)^{1/2} \\
            & = \left(\int \rho^2 d\mathcal{H}^2 \right)^{1/2} \left( \int_0^b \int_{\gamma_t} 1 \right)^{1/2}
            \leq \left(\int \rho^2 d\mathcal{H}^2 \right)^{1/2} \left(b \cdot \nu(t) 2^{-j}\right)^{1/2}.
        \end{align*}
        This implies that 
        \[ \frac{b\cdot 2^{j}}{20 \nu(t) } = \frac{1}{20 \cdot 2^7 \cdot \nu(t)} \leq \int \rho^2\,d\mathcal{H}^2.\]
        Since $\rho$ is arbitrary, this verifies the lemma with $\varphi(t) = (20 \cdot 2^6 \cdot \nu(t))^{-1}$.
    \end{proof}

	\bibliographystyle{abbrv}  
	\bibliography{biblio}

\end{document}